\documentclass[conference] {IEEEtran}
\usepackage[utf8]{inputenc}
\usepackage[reqno]{amsmath}
\usepackage{amssymb}
\usepackage{amsthm}
\usepackage{mathtools}
\usepackage{paralist}

\usepackage{todonotes}

\newcommand{\R}{\mathbb{R}}
\newcommand{\Q}{\mathbb{Q}}
\newcommand{\N}{\mathbb{N}}
\newcommand{\Z}{\mathbb{Z}}

\let\emptyset\varnothing

\usepackage{dsfont}

\newcommand{\eps}{\varepsilon}
\newcommand{\Realization}{\operatorname{R}}

\newcommand{\Indicator}{\mathds{1}}
\newcommand{\FirstN}[1]{\underline{#1}}

\newcommand{\lowerBound}{a_N^{\omega,-}}
\newcommand{\upperBound}{a_N^{\omega,+}}

\newtheorem{lemma}{Lemma}[section]
\newtheorem{proposition}[lemma]{Proposition}
\newtheorem{theorem}[lemma]{Theorem}

\theoremstyle{definition}
\newtheorem{definition}[lemma]{Definition}

\theoremstyle{remark}

\newtheorem*{remark*}{Remark}

\numberwithin{equation}{section}

\begin{document}

\title{Approximation in $L^p(\mu)$\\ with deep ReLU neural networks}

\author{
\IEEEauthorblockN{Felix {\sc Voigtlaender}}
\IEEEauthorblockA{Katholische Universität Eichstätt--Ingolstadt\\
Ostenstraße 26, 85072 Eichstätt, Germany\\
\texttt{felix@voigtlaender.xyz}}

\and \IEEEauthorblockN{Philipp {\sc Petersen}}
\IEEEauthorblockA{
Mathematical Institute, University of Oxford\\
OX2 6GG, Oxford, UK\\
\texttt{pc.petersen.pp@gmail.com}}
}

\maketitle

\begin{abstract}
  We discuss the expressive power of neural networks which use the non-smooth
  ReLU activation function \({\varrho(x) = \max\{0,x\}}\)
  by analyzing the approximation theoretic properties of such networks.
  The existing results mainly fall into two categories:
  approximation using ReLU networks with a \emph{fixed} depth,
  or using ReLU networks whose depth \emph{increases} with the approximation accuracy.
  After reviewing these findings, we show that the results concerning networks with fixed depth---%
  which up to now only consider approximation in \(L^p(\lambda)\) for the Lebesgue measure \(\lambda\)---%
  can be generalized to approximation in \(L^p(\mu)\),
  for any finite Borel measure \(\mu\).
  In particular, the generalized results apply in the usual setting of statistical learning theory,
  where one is interested in approximation
  in \(L^2(\mathbb{P})\), with the probability measure \(\mathbb{P}\) describing the distribution of the data.
\end{abstract}

\section{Introduction}
\label{sec:Introduction}

In recent years, machine learning techniques based on deep neural networks
have significantly advanced the state of the art in applications like image classification,
speech recognition, and machine translation.
The networks used for such applications tend to use the non-smooth ReLU activation function
${\varrho(x) = \max\{0,x\}}$, since it is empirically observed to improve the training procedure
\cite{DeepLearningNature}.

In this paper, we focus on the \emph{expressive power} of such neural networks.
Precisely, given a function class $\mathcal{F}$ and an approximation accuracy $\eps > 0$,
we aim to find a \emph{complexity bound} $N = N(\mathcal{F}, \eps)$
such that for any $f \in \mathcal{F}$, one can find a ReLU network $\Phi_{\eps}^f$
of complexity at most $N$ satisfying ${\|f - \Phi_{\eps}^f\| \leq \eps}$.
Here, the \emph{complexity} of the network is measured in terms of its \emph{depth}
(the number of layers) and in terms of the number of \emph{neurons} and \emph{weights}.
The approximation error will be either measured in the uniform norm or in $L^p(\mu)$
for some measure $\mu$.
When we simply write $L^p$, it is understood that $\mu = \lambda$ is taken
to be the Lebesgue measure.

\paragraph*{Structure of the paper}

We start by reviewing existing results which provide complexity bounds $N(\mathcal{F}, \eps)$
for approximating functions from the class $\mathcal{F} = \mathcal{F}_{d,\beta,B}$
of all $C^\beta$ functions $f$ on $Q := Q_d := [-\tfrac{1}{2}, \tfrac{1}{2}]^d$
that satisfy $\|f\|_{C^\beta} \leq B$.
These results fall into two categories:
The first considers approximation in $L^p$ using ReLU networks of \emph{fixed} depth,
while the second considers \emph{uniform} approximation using networks of \emph{increasing} depth.
We also present a novel result, showing that the complexity bounds of the first category also apply for
approximation in $L^p(\mu)$; see Theorem~\ref{thm:MainTheorem}.

Note that if $N(\mathcal{F}, \eps)$ is a valid complexity bound, then so is any
$N'(\mathcal{F}, \eps) \geq N(\mathcal{F}, \eps)$.
Therefore, after reviewing the existing complexity bounds,
we also discuss their \emph{optimality}.

In the final section of the paper, we prove Theorem~\ref{thm:MainTheorem}.

\section{Approximation results using ReLU networks}
\label{sec:OverviewOfApproximationTheory}

In this section, we review the existing findings concerning the approximation properties
of ReLU networks.
In doing so, we first focus on approximation using ReLU networks with a fixed depth,
and then see what changes when the depth of the networks is allowed to grow
with the approximation accuracy.

First of all, however, we formally define neural networks
and discuss how to measure their complexity.
Here and in the remainder of the paper, we write $\FirstN{m} := \{1,\dots,m\}$.

\begin{definition}\label{def:NeuralNetwork}
  A \emph{neural network} $\Phi$ with $L = L(\Phi) \in \N$ layers, input dimension $d \in \N$
  and output dimension $k \in \N$ is a tuple $\Phi = \big( (A_1,b_1), \dots, (A_L, b_L) \big)$,
  where $A_\ell \in \R^{N_\ell \times N_{\ell - 1}}$ and $b_\ell \in \R^{N_\ell}$
  for $\ell \in \FirstN{L}$ and where $N_0 = d$ and $N_L = k$.

  Given $\varrho : \R \to \R$ (called the \emph{activation function}),
  the \emph{$\varrho$-realization} of $\Phi$ is the function
  $\Realization_\varrho (\Phi) : \R^d \to \R^k, x \mapsto x_L$, where
  $x_0 := x \in \R^{d}$ and $x_L := A_L \, x_{L-1} + b_L \in \R^{k}$, while\vspace{-0.1cm}
  \[
    x_\ell := \varrho (A_\ell \, x_{\ell - 1} + b_\ell) \in \R^{N_\ell}
    \quad \text{for} \quad \ell \in \FirstN{L-1},
    \vspace{-0.1cm}
  \]
  where $\varrho(y) = \big(\varrho(y_1), \dots, \varrho(y_n)\big)$ for $y = (y_1,\dots,y_n) \in \R^n$.

  The \emph{number of neurons} of $\Phi$ is ${N(\Phi) := \sum_{\ell = 0}^L N_\ell \in \N}$,
  while the \emph{number of (nonzero) weights} of $\Phi$ is given by
  $W(\Phi) := \sum_{i=1}^L \big( \|A_i\|_{\ell^0} + \|b_i\|_{\ell^0} \big)$,
  with $\|A\|_{\ell^0}$ denoting the number of nonzero entries of a matrix or vector $A$.

  Given $\Omega \subset \R$, we say that \emph{all weights of $\Phi$ belong to $\Omega$}
  if all entries of the matrices $A_1,\dots,A_L$ and the vectors $b_1,\dots,b_L$ belong to $\Omega$.
  Given $s \in \N$ and $\eps \in (0,\tfrac{1}{2})$, we say that the network $\Phi$ is
  \emph{$(s,\eps)$-quantized}, if all weights of $\Phi$ belong to the set
  $[-\eps^{-s}, \eps^{-s}] \cap 2^{-s \lceil \log_2 (1/\eps) \rceil} \Z$.
\end{definition}

Weight-quantization is a further notion of complexity, which---when combined with bounds
on the number of network weights---restricts the number of bits needed to encode the network.

In the remainder of the paper, we will only consider the ReLU activation function
$\varrho : \R \to \R, x \mapsto \max \{0,x\}$.

\subsection{$L^p$ approximation using fixed-depth networks}
\label{sub:FixedDepthUpperBounds}

The following is the main existing result concerning approximation of $C^\beta$ functions
using \emph{fixed-depth} ReLU networks.

\pagebreak

\begin{theorem}\label{thm:OurSmoothFunctionApproximationTheorem}%
  (\cite[Theorem A.9]{OurReLUPaper})
  Let $\beta,B,p \in (0,\infty)$, 
  ${d \in \N}$, and $Q := [-\tfrac{1}{2}, \tfrac{1}{2}]^d$.
  There are $C > 0$ and $s \in \N$ (depending on $d,\beta,B,p$)
  such that for $\eps \in (0,\tfrac{1}{2})$ and ${\vphantom{\vrule height 0.4cm}f \in C^\beta (Q)}$
  with $\|f\|_{C^\beta} \leq B \vphantom{\sum_j}$,
  there is an $(s,\eps)$-quantized network $\Phi_\eps^f$
  with $L(\Phi_\eps^f) \leq 11 + (1 + \lceil \log_2 \beta \rceil) (11 + \tfrac{\beta}{d})$,
  and with $\|f - \Realization_\varrho (\Phi_\eps^f)\|_{L^p} \leq C \eps$
  and $N(\Phi_\eps^f) \lesssim W(\Phi_\eps^f) \leq C \, \eps^{-d/\beta}$.
\end{theorem}

In Section~\ref{sec:MainProof}, we will prove the following generalization:

\begin{theorem}\label{thm:MainTheorem}
  Let $d \in \N$ and $\beta, B, p \in (0,\infty)$, and let $\mu$ be a finite Borel measure
  on ${Q := [-\tfrac{1}{2}, \tfrac{1}{2}]^d}$.
  There are $C > 0$ and $s \in \N$ (depending on $d,p,\beta,B,\mu$)
  such that for $\eps \in (0,\tfrac{1}{2})$ and
  $f \in C^\beta (Q)$ with ${\|f\|_{C^\beta} \leq B}$,
  there is an $(s,\eps)$-quantized network $\Phi_{\eps}^f$
  with $L(\Phi_{\eps}^f) \leq 7 + (1 + \lceil \log_2 \beta \rceil) (11 + \tfrac{\beta}{d})$,
  and
  \[
    \|f - \Realization_\varrho (\Phi_{\eps}^f)\|_{L^p(\mu)} \leq C \eps
    \quad\!\!\! \text{and} \quad\!\!\!
    N(\Phi_\eps^f) \lesssim W (\Phi_{\eps}^f) \leq C \, \eps^{-d/\beta} .
    \vspace{-0.1cm}
  \]
\end{theorem}


The optimality of the complexity bound $W(\Phi_\eps^f) \lesssim \eps^{-d/\beta}$
will be discussed in detail in Section~\ref{sec:Optimality}.
A related question concerns the optimality of the \emph{depth} of the networks.
The next result shows that---up to logarithmic factors---the depth of the networks
in Theorems~\ref{thm:OurSmoothFunctionApproximationTheorem} and \ref{thm:MainTheorem}
is indeed optimal.

\begin{theorem}\label{thm:SmoothFunctionLayerApproximationRate}%
  (\cite[Theorem C.6]{OurReLUPaper}; see \cite{SafranShamir} for the case $p=2$)

  Let $\emptyset \neq \Omega \subset \R^d$ be open, bounded, and connected.
  Let $f \in C^3(\Omega)$ be nonlinear and $p \in (0,\infty)$.
  There is $C_{f,p,d} > 0$ such that for any neural network $\Phi$ of depth $L(\Phi)$, we have
  \vspace{-0.15cm}
  \[
    \|f - \Realization_\varrho (\Phi)\|_{L^p}
    \geq C_{f,p,d} \cdot \big( 1 + \min\{ N(\Phi), W(\Phi) \} \big)^{-2 L(\Phi)} .
  \]
\end{theorem}

Thus, to attain $\|f - \Realization_\varrho (\Phi^f_\eps)\|_{L^p} \lesssim \eps$
subject to the complexity bound $W(\Phi_\eps^f) \lesssim \eps^{-d/\beta}$,
the networks $\Phi_\eps^f$ must satisfy $L(\Phi_\eps^f) \geq \beta/2d$,
at least for $\eps > 0$ small enough.

In a nutshell, these results show that \emph{ReLU networks achieve better approximation rates
for smoother functions. To attain these better rates, however, one has to use deeper networks.}

\paragraph*{Further results}

One can also derive $L^p$ approximation rates for a certain class of \emph{discontinuous} functions;
see \cite{OurReLUPaper}.
Further, the presented results for fully connected networks
are equivalent to approximation results for certain simplified \emph{convolutional} networks
\cite{OurCNNPaper} that do not employ pooling operations.

We close our tour of approximation results using fixed depth networks
with the following result.

\begin{proposition}\label{thm:YarotskyUniformConstantDepth}(see \cite[Proposition~1]{YarotskyOptimalApproximation})
  Let $d \in \N$ and $\beta \in (0,1]$, and $Q := [-\tfrac{1}{2}, \tfrac{1}{2}]^d$.
  There is $C = C(d,\beta) > 0$ such that for each $f \in C^\beta(Q)$ and $\eps \in (0,\tfrac{1}{2})$,
  there is a neural network $\Phi_\eps^f$ with $L = L(d)$ layers such that
  \vspace{-0.15cm}
  \[
    \|f - \Realization_\varrho (\Phi_\eps^f)\|_{\sup} \leq \eps \, \|f\|_{C^\beta}
    \!\! \quad \text{and} \quad \!\!
    N(\Phi^f_\eps) \! \lesssim \! W(\Phi_\eps^f) \! \leq \! \tfrac{C}{\eps^{d/\beta}} . 
  \]
\end{proposition}

Though not explicitly stated in \cite{YarotskyOptimalApproximation},
one can show that the same statement holds for certain $(s,\eps)$-quantized networks $\Phi_\eps^f$,
where $s = s(d,B)$ and $\|f\|_{C^\beta} \leq B$.

Such a uniform approximation is much stronger than approximation in $L^p(\mu)$.
It should be noted, however, that the above result only applies for the
``low smoothness, slow approximation regime'' $\beta \in (0,1]$
where a piecewise affine approximation yields an optimal error.
For $\beta > 1$ it is an open problem whether the bounds in
Theorems~\ref{thm:OurSmoothFunctionApproximationTheorem} and \ref{thm:MainTheorem}
also hold for \emph{uniform} approximation using fixed-depth networks.




\subsection{Uniform approximation using networks of growing depth}
\label{sub:GrowingDepthUpperBounds}


While it is open whether fixed-depth networks
can satisfy $\|f - \Realization_\varrho (\Phi_\eps^f)\|_{\sup} \lesssim \eps$
and $W(\Phi_\eps^f) \lesssim \eps^{-d/\beta}$
for $f \in C^\beta$ and $\beta > 1$,
this \emph{is} possible with a mild depth-growth as $\eps \to 0$.

\begin{theorem}\label{thm:YarotskyUniformApproximation}(see \cite[Theorem~1]{YarotskyErrorBounds})
  Let $d , k \in \N$ and set $Q := [-\tfrac{1}{2}, \tfrac{1}{2}]^d$.
  There is $C = C(d,k) > 0$ such that for any $\eps \in (0,\tfrac{1}{2})$ and $f \in C^k (Q)$
  with ${\|f\|_{C^k} \leq 1}$, there is a network $\Phi_\eps^f$
  satisfying
  \(
    \vphantom{\vline height 0.4cm}
    N(\Phi_\eps^f) \lesssim W(\Phi_\eps^f) \leq C \big( 1 + \ln(1/\eps) \big) \, \eps^{-d/k}
  \)
  and
  $\vphantom{\vline height 0.4cm}
  \|f - \Realization_\varrho (\Phi_\eps^f)\|_{\sup} \leq \eps$,
  as well as $L(\Phi_\eps^f) \leq C \big( 1 + \ln(1/\eps) \big)$.
\end{theorem}
Although not stated explicitly in \cite[Theorem~1]{YarotskyErrorBounds},
the proof shows that $\Phi_\eps^f$ can be chosen to be $(s,\eps)$-quantized
for some $s = s(d,k) \in \N$.

We close this section with a surprising result from \cite{YarotskyOptimalApproximation}.
In that paper, Yarotsky shows that if one does \emph{not} restrict the growth of the depth as $\eps \to 0$,
and if one does \emph{not} insist that the networks be quantized,
then one can \emph{significantly} beat the approximation rates stated in
Theorems~\ref{thm:YarotskyUniformApproximation}, \ref{thm:OurSmoothFunctionApproximationTheorem},
and \ref{thm:MainTheorem}---at least in the ``low smoothness'' regime $\beta \in (0,1]$:

\begin{theorem}\label{thm:YarotskySurprise}(see \cite[Theorem~2]{YarotskyOptimalApproximation})

  Let $d \in \N$, $\beta \in (0,1]$, and $Q := [-\tfrac{1}{2}, \tfrac{1}{2}]^d$.
  There is $C = C(d,\beta) > 0$ such that for any $\eps \in (0,\tfrac{1}{2})$
  and $f \in C^\beta (Q)$ with $\|f\|_{C^\beta} \leq 1$ 
  there is a network $\vphantom{\vline height 0.35cm}\Phi_\eps^f$ satisfying
  $\vphantom{\vline height 0.35cm}\|f - \Realization_\varrho (\Phi_\eps^f)\|_{\sup} \leq \eps$
  and $W(\Phi_\eps^f) \leq C \, \eps^{-d/(2\beta)}$.
\end{theorem}

Note that Theorem~\ref{thm:YarotskyUniformApproximation} only yields
$W(\Phi_\eps^f) \lesssim \eps^{-d/\beta}$ instead of $W(\Phi_\eps^f) \lesssim \eps^{-d/(2\beta)}$.
Also note that Theorem~\ref{thm:YarotskySurprise} does \emph{not} claim that the networks
can be chosen to be quantized.
In fact, this is \emph{impossible}, as shown in the next section.


\section{Optimality of the approximation results}
\label{sec:Optimality}

Assuming that the complexity of the individual weights of the network does not grow too
quickly as $\eps \downarrow 0$, the complexity bound $W(\Phi_\eps^f) \lesssim \eps^{-d/\beta}$ derived in
Theorems~\ref{thm:OurSmoothFunctionApproximationTheorem} and \ref{thm:YarotskyUniformApproximation}
is optimal up to a log-factor.
In fact, the same arguments as in the proof of \cite[Theorem 4.3]{OurReLUPaper} show the following:

\begin{proposition}\label{prop:OptimalityQuantizedWeights}
  Let $d,s \in \N$ and $\beta, B, p \in (0,\infty)$ and let $Q := [-\tfrac{1}{2}, \tfrac{1}{2}]^d$.
  Then there is a function $f \in C^\beta (Q)$ with $\|f\|_{C^\beta} \leq B$ and a null-sequence
  $(\eps_k)_{k \in \N}$ such that
  \begin{align*}
    & \inf \left\{
             W(\Phi)
             \,\,\colon\!
             \begin{array}{l}
               \Phi \text{ an $(s,\eps_k)$-quantized neural netw.} \\
               \text{and } \|f - \Realization_\varrho (\Phi)\|_{L^p} \leq \eps_k
             \end{array}
           \!\!\!\right\} \\
    & \geq \eps_k^{-d/\beta} \big/ \big( \log(1/\eps_k) \cdot \log(\log(1/\eps_k)) \big)
    \quad \forall \, k \in \N \, .
  \end{align*}
\end{proposition}

In particular, while Theorem~\ref{thm:MainTheorem} is not optimal for \emph{every} measure
$\mu$, it is optimal for the Lebesgue measure $\mu = \lambda$.

Proposition~\ref{prop:OptimalityQuantizedWeights} also shows that the
networks in Theorem~\ref{thm:YarotskySurprise}
\emph{can not be chosen to be $(s,\eps)$-quantized}.

The proof of Proposition~\ref{prop:OptimalityQuantizedWeights} is \emph{information-theoretic}:
On the one hand, \cite[Lemma B.4]{OurReLUPaper} shows that there are at most
$2^{\mathcal{O}(\eps^{-\theta} \log_2 (1/\eps))}$ many different realizations of
$(s,\eps)$-quantized ReLU networks that have $c \, \eps^{-\theta}$ nonzero weights.
On the other hand, \cite{ClementsEntropyNumbers} yields lower bounds
for the cardinality of families that are $\eps$-dense (with respect to the $L^p$ norm)
in the set of all $C^\beta$-functions $f$ such that $\|f\|_{C^\beta} \leq B$.

Finally, using bounds for the VC-dimension of neural networks (see \cite{VCDimensionBounds}),
Yarotsky showed that the approximation rates derived in Theorem~\ref{thm:YarotskySurprise} are optimal;
see \cite[Theorem~1]{YarotskyOptimalApproximation}.

In summary, we note the following:
\begin{compactitem}
  \item For quantized networks, the rates in
        Theorems~\ref{thm:OurSmoothFunctionApproximationTheorem}
        and \ref{thm:YarotskyUniformApproximation} are optimal.

  \item For fixed-depth networks, the rates \emph{for uniform approximation}
        in Theorem~\ref{thm:YarotskyUniformApproximation} cannot be improved,
        even \emph{without} assuming quantized networks; see \cite[Part 2 of Theorem 4]{YarotskyErrorBounds}.
        Note, however, that it is open whether these rates \emph{can be attained at all}
        using fixed-depth networks if $\beta > 1$.

  \item Since the VC dimension arguments used for proving \cite[Theorem 4]{YarotskyErrorBounds}
        do not seem to generalize to $L^p$-approximation,
        it is open whether the rates in Theorem~\ref{thm:OurSmoothFunctionApproximationTheorem}
        are optimal for bounded-depth networks \emph{without} assuming quantized networks.

  \item If one neither assumes bounded depth nor quantized networks,
        then the results in Theorems~\ref{thm:OurSmoothFunctionApproximationTheorem} and \ref{thm:YarotskyUniformConstantDepth}
        can be improved, at least for $\beta \in (0,2)$.
        The optimal rates for this setting and for $\beta \in (0,1]$ are given
        by Theorem~\ref{thm:YarotskySurprise}.
        It is open what the optimal rates for $\beta > 1$ are if one neither assumes bounded depth
        nor quantized networks.
\end{compactitem}

\section{Proof of Theorem~\ref{thm:MainTheorem}}
\label{sec:MainProof}

We begin with the following lemma which shows that ReLU networks can (approximately)
localize a function to a cube.

\begin{lemma}\label{lem:Localization}(modification of \cite[Lemma A.6]{OurReLUPaper})

  For $a,b \in \R^d$ and $0 < \eps < \min_{i \in \FirstN{d}} \tfrac{1}{2}(b_i - a_i)$,
  set\vspace{-0.1cm}
  \[
    [a,b) := \textstyle{\prod_{i=1}^d} [a_i,b_i)
    \quad \text{and} \quad
    [a,b]_\eps := \textstyle{\prod_{i=1}^d} [a_i + \eps, b_i - \eps].
  \]
  If $B \geq 1$, there is a $4$-layer network $\Lambda_{\eps,B}^{(a,b)}$
  with $1$-dimensional output and $(d+1)$-dimensional input,
  with ${W(\Lambda_{\eps,B}^{(a,b)}) \leq c(d)}$ weights,
  all of which have their absolute values bounded by
  ${d + B + \eps^{-1} \cdot (1 + \|a\|_{\ell^\infty} + \|b\|_{\ell^\infty})}$,
  and such that if $x \in \R^d$ and $y \in [-B,B]$, then
  \begin{equation}
    \big|
      \Realization_\varrho \! \big( \Lambda_{\eps,B}^{(a,b)} \big) (x,y)
      - y \, \Indicator_{[a,b)}(x)
    \big|
    \leq 2B \, \Indicator_{[a,b) \setminus [a,b]_\eps} (x).
    \label{eq:LocalizationEstimate}
  \end{equation}
\end{lemma}

\begin{proof}
  For $i \in \FirstN{d}$, define a function $t_i : \R \to \R$ by setting
  \({
    t_i (x) = \varrho(\tfrac{x - a_i}{\eps})
              - \varrho(\tfrac{x - a_i - \eps}{\eps})
              - \varrho(\tfrac{x - b_i + \eps}{\eps})
              + \varrho(\tfrac{x - b_i}{\eps})
  }\).
  Note that $t_i$ is the realization of a two-layer ReLU network with at most $12$
  nonzero weights, all of which satisfy the required bound on the absolute value.
  It is easy to see that $0 \leq t_i \leq 1$ and $t_ i \equiv 1$ on $[a_i + \eps, b_i - \eps]$,
  while $t_i \equiv 0$ on $\R \setminus (a_i, b_i)$.

  Now, define $T : \R^d \times \R \to \R$ by\vspace{-0.2cm}
  \[
    T(x,y) := \sum_{\ell = 0}^1
                (-1)^\ell B \,
                \varrho \Big(
                          \varrho \big( (-1)^\ell \tfrac{y}{B} \big)
                          - d
                          + \textstyle{\sum_{i=1}^d} t_i (x_i)
                        \Big).
    \vspace{-0.2cm}
  \]
  By construction, $T = \Realization_\varrho(\Lambda_{\eps,B}^{(a,b)})$
  for a ReLU network $\Lambda_{\eps,B}^{(a,b)}$ as in the statement of the lemma.
  Furthermore, if $y \in [-B,B]$, then
  $\varrho \big( (-1)^\ell \tfrac{y}{B} \big) - d + \textstyle{\sum_{i=1}^d} t_i (x_i) \leq 1$,
  which shows $|T(x,y)| \leq B$.
  Next, if $y \in [-B,B]$ and $x \in [a,b]_\eps$, then $t_i(x_i) = 1$ for all
  $i \in \FirstN{d}$, and hence $T(x,y) = y$.
  Finally, if $y \in [-B,B]$ and $x \in \R^d \setminus [a,b)$, then $t_i (x_i) = 0$
  for some ${i \in \FirstN{d}}$, which entails $T(x,y) = 0$.
  This proves Eq.~\eqref{eq:LocalizationEstimate}.
\end{proof}

In \cite{OurReLUPaper}, it was used that Estimate~\eqref{eq:LocalizationEstimate}
shows that $\Realization_\varrho(\Lambda_{\eps,B}^{(a,b)})(\bullet, f(\bullet))$
and $\Indicator_{[a,b]} \cdot f$ are close in $L^p$.
Our next result shows that \emph{if one properly chooses the endpoints $a,b$},
one even gets closeness (including an approximation rate) in $L^p(\mu)$.
In addition to several results from \cite{OurReLUPaper}, this observation is the main ingredient
for our proof of Theorem~\ref{thm:MainTheorem}.

\begin{proposition}\label{prop:PartitionOfUnityEndPoints}
  Let $d \in \N$, $\gamma > 0$, $p \in [1,\infty)$, and define $k := d + 1 + p \gamma$.
  For $N \in \N$, let ${\Omega_N := \Z^d \cap [0, 2 \cdot 2^N - 1]^d}$.
  For $a \in \R^d$ and $\omega \in \Omega_N$, let ${\lowerBound := a + \tfrac{\omega}{2^N}\in \R^d}$
  and ${\upperBound := a + \tfrac{\omega + (1,\dots,1)}{2^N} \in \R^d}$,
  as well as ${I_N^{a,\omega} := [\lowerBound, \upperBound)}$.

  Let $\mu$ be a finite Borel measure on
  $\vphantom{\vrule height 0.4cm}Q := [-\tfrac{1}{2}, \tfrac{1}{2}]^d$.
  For Lebesgue-almost every $a \in [-5,5]^d$, there is a constant ${C_a = C_a (a,p,\mu,\gamma) > 0}$
  such that if $B \geq 1$ and if for each $\omega \in \Omega_N$
  a measurable function $f_\omega : Q \to [-B,B]$ is given, then for every $N \in \N$,
  with $\Lambda_{\eps,B}^{(a,b)}$ as in Lemma~\ref{lem:Localization}, we have
  \[
    \bigg\| \!
      \sum_{\omega \in \Omega_N} \!\!\!
        \Big[ \!
          \Realization_\varrho (\Lambda_{2^{-kN}, B}^{(\lowerBound \!, \upperBound)})(\bullet, f_\omega(\bullet))
          - \Indicator_{I_N^{a,\omega}} \cdot f_\omega
        \Big]
    \bigg\|_{L^p(\mu)} \!\!\!
    \leq \frac{C_a B}{2^{N \gamma}}.
  \]
\end{proposition}

\begin{proof}
  We consider $\mu$ as a Borel measure on $\R^d$, by setting $\mu(A) := \mu(A \cap Q)$.
  If $x \in [\lowerBound, \upperBound) \setminus [\lowerBound, \upperBound]_{2^{-kN}}$,
  then $x_i \notin [(\lowerBound)_i + 2^{-kN}, (\upperBound)_i - 2^{-kN}]$ for some $i \in \FirstN{d}$;
  hence,
  $x_i \in [(\lowerBound)_i , (\lowerBound)_i + 2^{-kN}] \cup [(\upperBound)_i - 2^{-kN}, (\upperBound)_i]$.
  By definition of $a_N^{\omega,\pm}$, this implies $a_i \in J_N^{(i)} (x,\omega)$, where
  \[
    J_N^{(i)}(x,\omega)
    \!:=\! [(\theta_N^{x,\omega})_i - \tfrac{1}{2^{kN}}, (\theta_N^{x,\omega})_i]
          \cup [(\eta_N^{x,\omega})_i, (\eta_N^{x,\omega})_i + \tfrac{1}{2^{kN}}],
  \]
  with $\theta_N^{x,\omega} := x - 2^{-N} \omega$
  and $\eta_N^{x,\omega} := x - 2^{-N}(\omega + (1,\dots,1))$.

  \smallskip{}

  Let $R := [-5,5]^d$.
  Using $\mu(A) = \int_{\R^d} \Indicator_A (x) \, d\mu(x)$, 
  we see
  {\setlength{\belowdisplayskip}{4pt}
   \setlength{\abovedisplayskip}{2.5pt}
  \begin{align*}
    \circledast
    & := \!\!
         \int_{R}
          \sum_{N=1}^\infty \!
            2^{N p\gamma} \!\!
            \sum_{\omega \in \Omega_N} \!\!
              \mu \big( [\lowerBound, \upperBound) \!\setminus\! [\lowerBound, \upperBound]_{2^{-kN}} \big)
         d a \\
    & \leq \sum_{N=1}^\infty \,
             \sum_{\omega \in \Omega_N} \,
               \sum_{i=1}^d
                 2^{N \gamma p} \!\!
                 \underbracket{
                   \int_{R}
                     \int_{\R^d} \!
                       \Indicator_{J_N^{(i)}(x,\omega)} (a_i)
                     \, d\mu(x)
                   \, d a
                 }_{:= \oplus_{N}^{(i)}(\omega)} \! .
  \end{align*}
  }%
  Using the Lebesgue measure $\lambda$, Fubini's theorem shows
  {\setlength{\belowdisplayskip}{4pt}
   \setlength{\abovedisplayskip}{4pt}
  \[
    \oplus_N^{(i)}(\omega)
    =\!\! \int_{\R^d} \!\!
            \lambda \Big( \big\{ a \in  [-5,5]^d \colon a_i \in J_N^{(i)}(x,\omega) \big\} \Big)
          \, d\mu(x),
  \]}%
  from which we get $\oplus_N^{(i)}(\omega) \leq 2 \cdot 10^{d-1} \mu(\R^d) \cdot 2^{-Nk}$.
  Thus, 
  \({
    \vphantom{\vline height 0.4cm}
    \circledast
    \lesssim \sum_{N=1}^\infty |\Omega_N| \cdot 2^{N \gamma p} \cdot 2^{-N k}
    \lesssim \sum_{N=1}^\infty 2^{N (d + \gamma p - k)}
    <        \infty
  }\),
  since $k = 1 + d + \gamma p$.

  Recalling the definition of $\circledast$, we see that
  \(
    \sum_{N=1}^\infty \!\!
      \big[
        2^{N p \gamma} \!
        \sum_{\omega \in \Omega_N}
          \mu \big( [\lowerBound, \upperBound) \setminus [\lowerBound, \upperBound]_{2^{-kN}} \big)
      \big]
    \! < \! \infty
  \)
  for Lebesgue-almost every $a \in [-5,5]^d$.
  In particular, for Lebesgue-almost every $a \in [-5,5]^d$, there is a constant
  $C_a = C_a(a,p,\gamma,\mu) > 0$ such that for every $N \in \N$, we have
  \(
    \sum_{\omega \in \Omega_N}
      \mu \big( [\lowerBound, \upperBound) \setminus [\lowerBound, \upperBound]_{2^{-kN}} \big)
    \leq C_a \cdot 2^{-N p \gamma} 
  \).

  Let us fix such a point $a \in [-5,5]^d$, and for each $\omega \in \Omega_N$, let
  $f_\omega : Q \to [-B,B]$ be measurable.
  Estimate \eqref{eq:LocalizationEstimate} shows
  \begin{align*}
    & \sum_{\omega \in \Omega_N}
        \big|
          \Realization_\varrho \big( \Lambda_{2^{-kN}, B}^{(\lowerBound, \upperBound)} \big)
              (\bullet, f_\omega(\bullet))
          - \Indicator_{[\lowerBound, \upperBound)} \cdot f_\omega
        \big| \\
    & \leq 2B \, \sum_{\omega \in \Omega_N}
                   \Indicator_{[\lowerBound, \upperBound) \setminus [\lowerBound,\upperBound]_{2^{-kN}}}
      =    2B \cdot \Indicator_{P}
  \end{align*}
  for
  \(
    P
    := \biguplus_{\omega \in \Omega_N}
         [\lowerBound, \upperBound) \setminus [\lowerBound,\upperBound]_{2^{-kN}}
  \),
  where the union is disjoint.
  Hence,\vspace{-0.15cm}
  \begin{align*}
    & \bigg\| \!
        \sum_{\omega \in \Omega_N} \!\!\!
          \Big[ \!
            \Realization_\varrho (\Lambda_{2^{-kN}, B}^{(\lowerBound, \upperBound)})(\bullet, f_\omega(\bullet))
            - \Indicator_{[\lowerBound, \upperBound)} \cdot f_\omega
          \Big]
      \bigg\|_{L^p(\mu)} \\
    & \leq 2B \cdot \Big(
                      \sum_{\omega \in \Omega_N}
                        \mu \big(
                              [\lowerBound, \upperBound) \setminus [\lowerBound,\upperBound]_{2^{-kN}}
                            \big)
                    \Big)^{1/p} \\
    & \leq 2B \cdot C_a^{1/p} \cdot 2^{-N \gamma}
      \quad \forall \, N \in \N.
    \hspace{3.5cm}\text{\qedhere}
  \end{align*}
\end{proof}

\vspace{-0.1cm}
To complete the proof of Theorem~\ref{thm:MainTheorem}, we need two results from \cite{OurReLUPaper}.
The first result is concerned with an approximate implementation of a family of polynomials.
\vspace{-0.1cm}

\begin{lemma}\label{lem:PolynomImplementation}(see \cite[Lemma A.5]{OurReLUPaper})
  Let $d,m \in \N$ and $B, \beta > 0$.
  Set $Q := [-\tfrac{1}{2}, \tfrac{1}{2}]^d$,
  let $(x_\ell)_{\ell \in \underline{m}} \subset Q$,
  and $(c_{\ell,\alpha})_{\ell \in \underline{m}, \alpha \in \N_0^d, |\alpha| < \beta} \subset [-B,B]$.

  Then there are $c = c(d,\beta,B) > 0$, $s = s(d,\beta,B) \in \N$, and $L = L(d,\beta) \in \N$
  with $L \leq 1 + (1 + \lceil \log_2 \beta \rceil) (11 + \tfrac{\beta}{d})$ such that for all
  $\eps \in (0, \tfrac{1}{2})$, there is a neural network $\Phi_\eps$ with
  $d$-dimensional input, $m$-dimensional output, with
  $L(\Phi_\eps) \leq L$ and $W(\Phi_\eps) \leq c \cdot (m + \eps^{-d/\beta})$,
  such that all weights of $\Phi_\eps$ belong to $[-\eps^{-s}, \eps^{-s}]$, and such that\vspace{-0.1cm}
  \[
    \Big|
      [\Realization_\varrho (\Phi_\eps) (x)]_\ell
      - \sum_{|\alpha| < \beta} \!\! c_{\ell,\alpha} \, (x - x_\ell)^\alpha
    \Big|
    < \eps
    \quad \forall \, \ell \in \FirstN{m} \text{ and } x \in Q .
    \vspace{-0.1cm}
  \]
\end{lemma}

Our final ingredient is a consequence of Taylor's theorem.

\begin{lemma}\label{lem:HoelderTaylorApproximation}(see \cite[Lemma A.8]{OurReLUPaper})
  Let $n \in \N_0$, $\sigma \in (0,1]$, and $\beta = n + \sigma$.
  Let $d \in \N$ and $Q := [-\tfrac{1}{2}, \tfrac{1}{2}]^d$.
  There is a constant $C = C(\beta,d) > 0$ such that for each $f \in C^\beta(Q)$
  with $\|f\|_{C^\beta} \leq B$ and each $x_0 \in (-\tfrac{1}{2}, \tfrac{1}{2})^d$,
  there is a polynomial $p(x) = \sum_{|\alpha| \leq n} c_\alpha \, (x-x_0)^\alpha$
  with $c_\alpha \in [-C B, C B]$ and such that
  $|f(x) - p(x)| \leq CB \cdot |x - x_0|^\beta$ for all $x \in Q$.
\end{lemma}

%
%
%


\begin{proof}[Proof of Theorem~\ref{thm:MainTheorem}]
  Let us fix some $a \in \big( (-\tfrac{3}{2}, -\tfrac{1}{2}] \setminus \Q \big)^d$
  and a constant $C_1 = C_1 (a,p,\mu,\beta)> 0$ satisfying the conclusion of
  Proposition~\ref{prop:PartitionOfUnityEndPoints} for the choice $\gamma := \beta$.
  Such an $a$ exists, since $\big( (-\tfrac{3}{2}, -\tfrac{1}{2}] \setminus \Q \big)^d$
  has positive Lebesgue measure.

  \medskip{}

  Let $N := \lceil \log_2 \eps^{-1/\beta} \rceil \in \N$, whence
  ${\tfrac{1}{\eps^{1/\beta}} \leq 2^N \leq \tfrac{2}{\eps^{1/\beta}}}$.
  We observe that
  $\vphantom{\vrule height 0.4cm}{Q \subset a + [0,2)^d \subset \biguplus_{\omega \in \Omega_N} [\lowerBound, \upperBound)}$,
  since we have $Q - a \subset [0,2)^d$.
  Next, define ${\vphantom{\vrule height 0.4cm}Q^\circ = (-\tfrac{1}{2}, \tfrac{1}{2})^d}$ and
  ${\Omega_N^\ast := \{ \omega \in \Omega_N \colon Q \cap [\lowerBound, \upperBound) \neq \emptyset \}}$.
  Since ${a \in (\R \setminus \Q)^d}$, we see for each $\omega \in \Omega_N^\ast$
  that there is some $x_\omega \in Q^\circ \cap [\lowerBound, \upperBound)$.
  Set $m := |\Omega_N^\ast|$ and write $\Omega_N^\ast = \{\omega_1,\dots,\omega_m\}$ for suitable
  $\omega_1,\dots,\omega_m$. Note $m \leq |\Omega_N| = (2 \cdot 2^N)^d \leq 4^d \, \eps^{-d/\beta}$.
  For $i \in \FirstN{m}$, set $x_i := x_{\omega_i}$.

  \smallskip{}

  Let $f \in C^\beta(Q)$ with $\|f\|_{C^\beta} \leq B$.
  Lemma~\ref{lem:HoelderTaylorApproximation} yields for each $\ell \in \FirstN{m}$ a sequence
  \(
    (c_{\ell,\alpha})_{\alpha \in \N_0^d, |\alpha| < \beta}
    \subset [-C_2 B, C_2 B]
  \)
  such that $|f(x) - p_\ell(x)| \leq C_2 B \cdot |x - x_\ell|^\beta$ for all $x \in Q$,
  where $p_\ell (x) := \sum_{|\alpha| < \beta} c_{\ell,\alpha} \, (x - x_\ell)^\alpha$.
  Here, $C_2 = C_2 (d,\beta) > 0$.

  \smallskip{}

  Next, we apply Lemma~\ref{lem:PolynomImplementation} (with $C_2 B$ instead of $B$) to obtain a
  neural network $\Phi$ with $d$-dimensional input and $m$-dimensional output such that
  \(
    \big| \big[ \Realization_\varrho (\Phi) (x) \big]_\ell - p_\ell (x) \big|
    \leq \tfrac{\eps}{4}
    < 1
  \)
  for all $x \in Q$ and
  $\vphantom{\vrule height 0.4cm}L(\Phi) \leq 1 + (1 + \lceil \log_2 \beta \rceil) (11 + \beta/d)$,
  as well as
  \(
    \vphantom{\vrule height 0.4cm}
    W(\Phi)
    \leq C_3 \cdot \big( m + (\eps/4)^{-d/\beta} \big)
    \leq C_4 \cdot \eps^{-d/\beta} .
  \)
  Here, $C_i = C_i (d,\beta,B)$ for $i \in \{3, 4\}$.
  Finally, all weights of $\Phi$ belong to $[-\eps^{-s_1}, \eps^{-s_1}]$
  for some $s_1 = s_1 (d, \beta, B) \in \N$.

  \smallskip{}

  Next, we use $|p_\ell(x)| \leq |p_\ell(x) - f(x)| + |f(x)|$ to derive
  \vspace{-0.08cm}
  \[
    |p_\ell(x)|
      \leq C_2 B \, |x - x_\ell|^\beta + B
      \leq B (1 + d^\beta C_2),
    \vspace{-0.08cm}
  \]
  whence
  $|f_{\omega_\ell}| \leq \tfrac{\eps}{4} + B (1 + d^\beta C_2) \leq B' = B'(d,\beta,B) \geq 1$
  for $f_{\omega_\ell} := \big( \Realization_\varrho (\Phi) \big)_\ell |_Q$.
  Let us set $f_\omega \equiv 0$ for ${\omega \in \Omega_N \setminus \Omega_N^\ast}$.

  \smallskip{}

  Now, set $g := \sum_{\omega \in \Omega_N^\ast} \Indicator_{[\lowerBound, \upperBound)} \, f_\omega$
  and ${k := d + 1 + p \beta}$ and furthermore
  \(
    G
    := \sum_{\omega \in \Omega_N}
         \Realization_\varrho (\Lambda_{2^{-kN}, B'}^{(\lowerBound, \upperBound)})(\bullet, f_\omega(\bullet))
  \).
  Then Proposition \ref{prop:PartitionOfUnityEndPoints} (with $B'$ instead of $B$) shows
  \({
    \| G - g \|_{L^p(\mu)}
    \leq \frac{C_1 B}{2^{N \gamma}}
    \leq C_5 \, \eps
  }\),
  where $C_5 = C_5 (p,\mu,\beta,d,B)$.

  For $x \in Q$, we have $x \in [a_N^{\omega_\ell,-}, a_N^{\omega_\ell,+})$ for a unique
  ${\ell \in \FirstN{m}}$, whence
  \(
    \vphantom{\vline height 0.38cm}
    |x - x_\ell|
    \leq d \cdot \|x - x_\ell\|_{\ell^\infty}
    \leq d \, 2^{-N}
    \leq d \, \eps^{1/\beta}
  \).
  Therefore,
  \(
    \vphantom{\vline height 0.38cm}
    g(x) = f_{\omega_\ell}(x) = \big[\Realization_\varrho (\Phi) (x)\big]_\ell
  \),
  and hence
  {
  \setlength{\abovedisplayskip}{3pt}
  \setlength{\belowdisplayskip}{3pt}
  \begin{align*}
    |f(x) - g(x)|
    & \leq |f(x) - p_\ell (x)| + |p_\ell(x) - [\Realization_\varrho (\Phi) (x) ]_\ell| \\
    & \leq C_2 B \cdot |x - x_\ell|^\beta + \tfrac{\eps}{4}
      \leq C_6 \cdot \eps,
  \end{align*}
  where $C_6 = C_6 (d,\beta,B)$.}
  Since $\mu$ is finite, we thus see ${\|f - G\|_{L^p(\mu)} \leq C_7 \, \eps}$
  for $C_7 = C_7 (p,\mu,d,\beta,B)$.

  It remains to show
  \(
    \| G - \Realization_\varrho (\Phi_{\eps}^f) \|_{L^p(\mu)} \leq \eps
  \)
  for a network $\Phi_{\eps}^f$ as in the statement of Theorem~\ref{thm:MainTheorem}.
  First, since the class of neural networks is closed under composition and addition
  (including control over the complexity of the resulting networks;
  see the end of the proof of \cite[Lemma A.7]{OurReLUPaper} for details), we see that
  $G = \Realization_\varrho (\Psi_\eps^f)$ for a network $\Psi_\eps^f$
  with $W(\Psi_\eps^f) \leq C_8 \eps^{-d/\beta}$ and
  $L(\Psi_\eps^f) \leq 7 + (1 + \lceil \log_2 \beta \rceil) (11 + \tfrac{\beta}{d})$,
  where $C_8 = C_8 (p,d,\beta,B)$,
  and such that all weights of $\Psi_\eps^f$ lie in $[-\eps^{-s_2}, \eps^{-s_2}]$
  for some $s_2 = s_2 (p,d,\beta,B,\mu) \in \N$.
  Here, we used that $|\Omega_N| \lesssim 2^{dN} \lesssim \eps^{-d/\beta}$
  and that all weights of the networks $\Lambda^{(\lowerBound,\upperBound)}_{2^{-kN}, B'}$
  have absolute value at most
  \({
    d + B' + 2^{kN} (1 + \|\lowerBound\|_{\ell^\infty} + \|\upperBound\|_{\ell^\infty})
    \leq d + B' + \tfrac{15 \cdot 2^k}{\eps^{k/\beta}},
  }\)
  while all weights of $\Phi$ have absolute value at most $\eps^{-s_1}$.
  Finally, \cite[Lemma~3.7]{ApproximationWithSparselyConnectedDNN} can be used to obtain a
  quantized network.
  Precisely, that lemma yields a network $\Phi_\eps^f$
  with the properties stated in Theorem~\ref{thm:MainTheorem} and such that
  $\|\Realization_\varrho (\Phi_\eps^f) - \Realization_\varrho (\Psi_\eps^f)\|_{\sup} \leq \eps$.
\end{proof}

\end{document}